\documentclass[a5paper, 10pt]{article}

\usepackage{cite, amsmath, amssymb, booktabs, lastpage, graphicx}
\usepackage[hidelinks]{hyperref}

\usepackage{geometry}
\usepackage{setspace}

\usepackage{amsthm}

\theoremstyle{plain}
\newtheorem{lemma}{Lemma}
\newtheorem{theorem}{Theorem}

\theoremstyle{remark}

\theoremstyle{definition}

\newtheorem{example}{Example}

\usepackage{graphicx}
\makeatletter
\renewcommand{\maketitle}{
	\begin{center}
    \rule[.2em]{\textwidth}{0.353mm}
		\begin{minipage}[m]{0.35\textwidth}
			{\scriptsize
				\begin{center}
					
			\end{center}}
		\end{minipage}\hfill
		\begin{minipage}[m]{0.65\textwidth}
			
		\end{minipage}
		\rule[1em]{\textwidth}{.353mm}
		\baselineskip=0.30in
		{\Large\bfseries \@title} \par
		\vspace{5mm}
		\baselineskip=0.2in
		{\large\bfseries \@author}\par
		\vspace{1mm}
		{\it \@address} \par
		{\small\tt \@email} \par
		\vspace{3mm}
		{\small (Received \@date)} \par
	\end{center}
	\vspace{3mm}
}
\newcommand{\address}[1]{\def\@address{#1}}
\newcommand{\email}[1]{\def\@email{#1}}

\makeatother

\newcommand{\acknowledgment}[1]{\vspace{5mm}\singlespacing
	{\noindent\textbf{\textit{Acknowledgment\/}:} #1}
}

\usepackage{fancyhdr}

\usepackage[margin=1cm,%
			font=footnotesize,%
			format=hang,%
			labelsep=period,%
			labelfont=bf]{caption}

\newcommand{\NN}{\mathbb{N}}

\newcommand{\ZZ}{\mathbb{Z}}

\newcommand{\diam}{\mathrm{diam}}
\newcommand{\degr}{\mathcal{D}}
\newcommand{\Tr}{\mathrm{Tr}}
\newcommand{\rk}{\mathrm{rk}}
\newcommand{\Ann}{\mathrm{Ann}}
\newcommand{\Annl}{\mathrm{Ann}_L}
\newcommand{\Annr}{\mathrm{Ann}_R}
\newcommand{\DEF}[1]{\emph{#1}}

\title{The Wiener index and the Wiener Complexity of the zero-divisor graph of a ring}
\author{David Dol\v zan$^{a,}$}

\address{$^a$D.~Dol\v zan:~Department of Mathematics, Faculty of Mathematics
and Physics, University of Ljubljana, Jadranska 19, SI-1000 Ljubljana, Slovenia, and Institute of Mathematics, Physics and Mechanics, Jadranska 19, SI-1000 Ljubljana, Slovenia}

\email{david.dolzan@fmf.uni-lj.si}

\date{\today}

\begin{document}

\maketitle

\begin{abstract}
We calculate the Wiener index of the zero-divisor graph of a finite semisimple ring. We also calculate the Wiener complexity of the zero-divisor graph of a finite simple ring and find an upper bound for the Wiener complexity in the semisimple case.
\end{abstract}

\onehalfspacing

 \section{Introduction}

\bigskip

 The Wiener index of a graph, introduced by Wiener in \cite{Wiener}, turns out to be among the most important of the graph indices. It is defined as $W(\Gamma)=\frac{1}{2}\sum_{u,v \in V(\Gamma)}d(u,v)$. Wiener discovered that the boiling points of alkanes and the sum of the distances between any pair of vertices in its molecular structure graph have a close connection. Later, a strong correlation between the Wiener index and the chemical properties of a compound was found - for example, critical points in general (see \cite{stiel}), the density, the surface tension, and viscosity of compounds liquid phase (see \cite{rouvray}) and the van der Waals surface area of the molecule (see \cite{gutman}). Nowadays, this index is used for the preliminary screening of drug molecules \cite{agrawal}.  See also \cite{MR1843259} and the extensive list of references therein, as well as \cite{MR3381136} for the many applications of the Wiener index in chemistry. 

A related notion to the Wiener index is the Wiener complexity of a graph which was defined in \cite{Klavzar1, Klavzar2}. It denotes the number of different transmissions of vertices in a graph and  its importance stems from location theory where vertices with extremal transmission are studied, since they are target locations for facilities (see \cite{bresar, skreko, smart}).

In this paper, we study the Wiener index and the Wiener complexity of certain graphs that arise from algebraic structures. Namely, our setting is the (undirected) zero-divisor graph of a finite (possibly non-commutative) ring. Zero-divisor graphs of various algebraic structures have proved to be very useful in revealing the algebraic properties through the graph-theoretical properties of the prescribed graphs. It all began in 1988, when Beck \cite{MR944156} introduced the notion of the zero-divisor graph of a commutative ring. Later, Anderson and Livingston \cite{MR1700509} slightly adjusted the definition of the zero-divisor graph in order to be able to investigate the zero-divisor structure of commutative rings. The definition was later extended to study the non-commutative rings in \cite{MR2201052} and \cite{MR2037657}. Different authors then further extended this concept to semigroups  \cite{MR1911724}, nearrings \cite{MR2177522} and semirings \cite{MR2460702}. Recently, Wiener indices of some graphs arising from algebraical structures have been studied (see  \cite{asirnew1} and \cite{asirnew2} for a survey of the recent results). For example, the Wiener index of the zero-divisor graph of a finite commutative ring was studied in \cite{MR4377155}, while in \cite{asir, MR4257055} it was studied in the special case of the ring of integers modulo $n$. Furthermore, the Wiener index of the total graph  of the ring of integers modulo $n$ was studied in \cite{MR3754839} and the Wiener index of the unit graph was studied in \cite{MR4414151}. The Wiener index of a compressed zero-divisor graph was also studied in \cite{MR4124061}.

This paper is structured as follows. In the next section, we gather the definitions and notations that we use throughout the paper. In Section 3, we calculate the Wiener index and the Wiener complexity of a zero-divisor graph of a simple finite ring (see Theorem \ref{mainsimple} and Theorem \ref{complexsimple}). In the final section, we use these results to study the zero-divisor graphs of semisimple finite rings. We calculate the Wiener index (see Theorem \ref{mainsemisimple}) and find an upper bound for the Wiener complexity (see Theorem \ref{complexsemisimple}).

\bigskip

 \section{Definitions and preliminaries}

\bigskip

For a finite ring $R$, we denote by $R^*$ the group of units in $R$ and by $Z(R)$ the set of zero-divisors in $R$, $Z(R)=\{x \in R;$ there exists
$0 \neq y \in R \text { such that } xy=0 \text { or } yx=0 \}$. 
We denote by $\Gamma(R)$ the (undirected) \emph{zero-divisor graph} of $R$, following the definition in \cite{MR2201052}: the vertex set $V(\Gamma(R))$ of $\Gamma(R)$ is the set of elements in $Z(R) \setminus \{0\}$ and 
an unordered pair of vertices $x,y \in V(\Gamma(R))$, $x \neq y$, is an edge 
$x-y$ in $\Gamma(R)$ if $xy = 0$ or $yx=0$.
The sequence of edges $x_0 - x_1$, $ x_1 - x_2$, ..., $x_{k-1} - x_{k}$ in a graph is called \emph{a path of length $k$}. The \DEF{distance} between vertices $x$ and $y$ is the length of the shortest
path between them, denoted by $d(x,y)$. 
The \DEF{diameter} $\diam(\Gamma)$ is the longest distance between any two vertices of the graph $\Gamma$.
The \DEF{transmission} of a vertex $u \in V(\Gamma)$ is defined by $\Tr(u)=\sum_{v \in V(\Gamma)}d(u,v)$.
The \DEF{Wiener index} of $\Gamma$ is defined as $W(\Gamma)=\frac{1}{2}\sum_{u,v \in V(\Gamma)}d(u,v)=\frac{1}{2}\sum_{u \in V(\Gamma)}\Tr(u)$.
Finally, the \DEF{Wiener complexity} of $\Gamma$ is defined as the number of all different transmissions of vertices in $V(\Gamma)$ and it is denoted by $C_W(\Gamma)$.

For a  commutative ring $R$, we denote by $M_n(R)$ the ring of all $n$-by-$n$ matrices with entries in $R$. We shall denote by $E_{ij} \in M_n(R)$ the matrix with $1$ at entry $(i,j)$ and zeros elsewhere.  For a matrix $A \in M_n(R)$, we shall denote by $A_{ij} \in R$ the $(i,j)$-th entry of $A$ and we shall denote the rank of matrix $A$ by $\rk(A)$.

For a finite ring $R$ and $x \in Z(R)$, we denote the degree of $x$ in $\Gamma(R)$ by $\degr(x)=|\{y \in R \setminus \{0,x\} ; yx= 0 \text { or } xy=0\}|$. We sometimes also denote $\degr(x)=0$ for $x \in R$ that is not a zero divisor.  Furthermore, we shall denote the annihilator of $x$ in $R$ by $\Ann(x)=\{y \in R; yx= 0 \text { or } xy=0\}$, as well as $\Annl(x)=\{y \in R; yx= 0\}$ and analogously, $\Annr(x)=\{y \in R; xy= 0\}$.

We shall limit ourselves to studying the zero-divisor graphs of finite rings due to the following Lemma.

\begin{lemma}(\cite{MR169870, MR207753})
\label{finite}
If $R$ is a ring with $m$ zero divisors, $2 \leq m < \infty$, then $R$ is a finite ring with $|R| \leq m^2$.
\end{lemma}

\bigskip

 \section{Simple rings}

\bigskip

In this section, we shall study the Wiener index and the Wiener complexity of a simple finite ring. 
We shall need the following lemmas.

\begin{lemma}(\cite[Corollary 4.1.]{MR2510978})
\label{diam}
Let $F$ be a field and $n \geq 2$. Then $\diam(\Gamma(M_n(F)))=2$.
\end{lemma}

\begin{lemma}(\cite[p. 6]{MR2217227})
\label{rankk} 
Let $F$ be a finite field of cardinality $q$ and $n \geq 2$. For every $k \in \{0,1,\ldots,n\}$ there are exactly $\prod_{j=0}^{k-1}{\frac{(q^n-q^j)^2}{q^k-q^j}}$ matrices in $M_n(F)$ of rank $k$.
\end{lemma}

This immediately yields the following.

\begin{lemma}
\label{novertices}
Let $F$ be a finite field of cardinality $q$ and $n \geq 2$. Then $|V(\Gamma(M_n(F)))|=|Z(M_n(F))|-1=q^{n^2}-q^{n \choose 2}\prod_{j=1}^{n}(q^{j}-1)-1$.
\end{lemma}
\begin{proof}
 By Lemma \ref{rankk}, there are exactly $\prod_{j=0}^{n-1}{(q^n-q^j)}=q^{n \choose 2}\prod_{j=0}^{n-1}(q^{n-j}-1)=q^{n \choose 2}\prod_{j=1}^{n}(q^{j}-1)$ invertible matrices in $M_n(F)$. Since all nonzero noninvertible matrices are in $V(\Gamma(M_n(F)))$, we immediately get the desired equation,
$|V(\Gamma(M_n(F)))|=q^{n^2}-q^{n \choose 2}\prod_{j=1}^{n}(q^{j}-1)-1$.
\end{proof}

The following lemma can be deduced from \cite{MR2201052}.

\begin{lemma}
\label{sizeofann}
Let $F$ be a finite field of cardinality $q$  and $n \geq 2$. Suppose $A \in M_n(F)$ is a matrix of rank $k$ for some $k \in \{1,2,\ldots,n\}$. Then $|\Ann(A)|=2q^{n(n-k)}-q^{(n-k)^2}$ and $\degr(A)=2q^{n(n-k)}-q^{(n-k)^2} - \epsilon$, where $\epsilon=1$ if $A^2 \neq 0$ and $\epsilon=2$ otherwise.
\end{lemma}
\begin{proof}
The statement follows directly from \cite[Lemma 14]{MR2201052}.
\end{proof}

Next, we shall need to find the number of square zero matrices of a specified rank.

\begin{lemma}
\label{sqzero}
Let $F$ be a finite field of cardinality $q$ and $n \geq 2$. Then $|\{0 \neq A \in M_n(F); A \text { is a matrix } \text {of rank } k \text { with } A^2=0\}|=\frac{\prod_{j=0}^{2k-1} \left( q^{n}- q^{j}\right )}{q^{k^2}\prod_{j=0}^{k-1}\left( q^{k}-q^{j} \right )}$.
\end{lemma}
\begin{proof}
If $A^2=0$ for some $A \in M_n(F)$ of rank $k$ then $\dim_F(A(F^n))=k$ and $\dim_F(A^2(F^n))=0$. We can now deduce from Remark 3.2 in \cite{MR2190192} that $|\{0 \neq A \in M_n(F); A \text { is a matrix of rank } k \text { with } A^2=0\}|={n \choose k}_q \prod_{j=1}^k{\left( q^{n-k}-q^{j-1} \right )}$, where${n \choose k}_q$ denotes the number of $k$-dimensional subspaces in $F^n$. 
To find the number of subspaces of dimension $k$, we have to count the number of ways one can choose $k$ independent vectors in $F^n$.
The first vector can be chosen in $q^n-1$ ways, the second vector can be chosen in $q^n-q$ ways, and so on.
However, some of these pickings generate the same vector space, so to arrive at the number of distinct vector spaces, we have to divide this by the number of different bases for a $k$ dimensional vector space, which we proceed to calculate by the same procedure as above. Therefore, 
 $|\{0 \neq A \in M_n(F); A \text { is a matrix of rank } k \text { with } A^2=0\}|=\frac{\left( q^{n}-1 \right )\left( q^{n}-q \right )\ldots \left( q^{n}-q^{k-1} \right )}{\left( q^{k}-1 \right )\left( q^{k}-q \right )\ldots \left( q^{k}-q^{k-1} \right )} \prod_{j=1}^k{\left( q^{n-k}-q^{j-1} \right )}=\frac{\prod_{j=0}^{2k-1} \left( q^{n}- q^{j}\right )}{q^{k^2}\prod_{j=0}^{k-1}\left( q^{k}-q^{j} \right )}$.
\end{proof}

Now, we have everything ready to prove the first theorem of this section.

\begin{theorem}
\label{mainsimple}
Let $F$ be a finite field of cardinality $q$ and $n \geq 2$. Then 
\begin{multline*}
W(\Gamma(M_n(F)))= \\
|Z(M_n(F))|^2 - \frac{5}{2}|Z(M_n(F))|+\frac{3}{2} + \\ 
\frac{1}{2} \sum_{k=1}^{n-1}\left(\prod_{j=0}^{k-1}{\frac{q^n-q^j}{q^k-q^j}}\left(\prod_{j=0}^{k-1}\left( q^{n-k}-q^j \right )-(2q^{n(n-k)}-q^{(n-k)^2})\prod_{j=0}^{k-1}(q^n-q^j)\right)\right), 
\end{multline*}
where $|Z(M_n(F))|=q^{n^2}-q^{n \choose 2}\prod_{j=1}^{n}(q^{j}-1)$.
\end{theorem}
\begin{proof}
For $i=1,2$ denote $D_i=|\{ (a,b); a,b \in V(\Gamma(M_n(F))) \text { such that } d(a,b)=i\}|$.
 We know by Lemma \ref{diam} that for every pair of (distinct) vertices $A, B \in V(\Gamma(M_n(F)))$, we either have $d(A,B)=1$ or $d(A,B)=2$, so $D_1+D_2=|V(\Gamma(M_n(F)))|(|V(\Gamma(M_n(F)))|-1)=(|Z(M_n(F))|-1)(|Z(M_n(F))|-2)$.
 Therefore $W(\Gamma(M_n(F)))=\frac{1}{2}(D_1+2D_2)=\frac{1}{2}(D_1+2((|Z(M_n(F))|-1)(|Z(M_n(F))|-2)-D_1))=
 (|Z(M_n(F))|-1)(|Z(M_n(F))|-2)-\frac{1}{2}D_1$.
 By using Lemmas \ref{rankk}, \ref{novertices}, \ref{sizeofann} and \ref{sqzero}, we arrive at  
\begin{multline*}
W(\Gamma(M_n(F)))=(|Z(M_n(F))|-1)(|Z(M_n(F))|-2)- \\ 
\frac{1}{2} \sum_{k=1}^{n-1}\prod_{j=0}^{k-1}{\frac{(q^n-q^j)^2}{q^k-q^j}}(2q^{n(n-k)}-q^{(n-k)^2})+\\
\frac{1}{2} (|Z(M_n(F))|-1)+\frac{1}{2}\sum_{k=1}^{n-1}\frac{\prod_{j=0}^{2k-1} \left( q^{n}- q^{j}\right )}{q^{k^2}\prod_{j=0}^{k-1}\left( q^{k}-q^{j} \right )} = \\
|Z(M_n(F))|^2 - \frac{5}{2}|Z(M_n(F))|+\frac{3}{2} + \\ 
\frac{1}{2} \sum_{k=1}^{n-1}\left(\prod_{j=0}^{k-1}{\frac{q^n-q^j}{q^k-q^j}}\left(\prod_{j=0}^{k-1}\left( q^{n-k}-q^j \right )-(2q^{n(n-k)}-q^{(n-k)^2})\prod_{j=0}^{k-1}(q^n-q^j)\right)\right)=\\
|Z(M_n(F))|^2 - \frac{5}{2}|Z(M_n(F))|+\frac{3}{2} + \\ 
\frac{1}{2} \sum_{k=1}^{n-1}\left(\prod_{j=0}^{k-1}{\frac{q^n-q^j}{q^k-q^j}}\left(\prod_{j=0}^{k-1}\left( q^{n-k}-q^j \right )-(2q^{n(n-k)}-q^{(n-k)^2})\prod_{j=0}^{k-1}(q^n-q^j)\right)\right).
\end{multline*}
Finally, by Lemma \ref{novertices}, we have $|Z(M_n(F))|=q^{n^2}-q^{n \choose 2}\prod_{j=1}^{n}(q^{j}-1)$.
\end{proof}

Not unexpectedly, the formula for the Wiener index is quite a convoluted one. To illustrate Theorem \ref {mainsimple}, let us examine the simpler case $n=2$ separately, where the formula is of course significantly more simple.

\begin{example}
Direct application of Theorem \ref{mainsimple} and some calculations that we shall omit here, immediately yield that
$W(\Gamma(M_2(F)))=|F|^6+|F|^5-\frac{3}{2}|F|^4-3|F|^3-\frac{1}{2}|F|^2+2|F|+1$. In particular for example, $W(\Gamma(M_2(\ZZ_2)))=51$.
\end{example}

As a final result in this section, we can completely determine the Wiener complexity of a zero divisor graph of a simple finite ring. It turns out that the Wiener complexity is dependent only on the size of matrices and not on the size of the underlying field. Specifically, we get the following theorem.

\begin{theorem}
\label{complexsimple}
Let $F$ be a finite field and $n \geq 2$. Then 
$$C_W(\Gamma(M_n(F)))=2(n-1).$$
\end{theorem}
\begin{proof}
Choose $0 \neq A \in Z(M_n(F))$. Then by Lemma \ref{diam}, we have $\Tr(A)=\degr(A)+2(|V(\Gamma(M_n(F)))|-1 - \degr(A))=2(|Z(M_n(F))|-2)-\degr(A)$. This implies that for $0 \neq A, B \in Z(M_n(F))$ we have $\Tr(A)=\Tr(B)$ if and only if $\degr(A)=\degr(B)$. 

Obviously, Lemma \ref{sizeofann} implies that if $\rk(A)=\rk(B)$ then $\degr(A)=\degr(B)$ if and only if either $A^2=B^2=0$ or $A^2, B^2 \neq 0$.

On the other hand, suppose that $\rk(A) \neq \rk(B)$. Denote $k=\rk(A)$, $l=\rk(B)$ and $q=|F|$ and observe that by Lemma \ref{sizeofann}, $\degr(A)=2q^{n(n-k)}-q^{(n-k)^2}-\epsilon$ and $\degr(B)=2q^{n(n-l)}- q^{(n-l)^2}-\varphi$ for some $\epsilon, \varphi \in \{1,2\}$. Therefore $2q^{n(n-k)}-2q^{n(n-l)}+ q^{(n-k)^2}-q^{(n-l)^2}=\epsilon-\varphi \in \{-1,0,1\}$. Since $1 \leq k,l \leq n-1$, the left hand of the equation is divisible by $q$, therefore $\epsilon - \varphi = 0$. Since $k \mapsto 2q^{n(n-k)}-q^{(n-k)^2}$ is a strictly decreasing function, we can conclude that $\degr(A) \neq \degr(B)$.

Since for any $k$ from $1$ to $n-1$, we can find a matrix in $M_n(F)$ that is square-zero as well as one that is not, we see that there are exactly $2(n-1)$ different transmissions in the zero-divisor graph. Thus, we have proved that $C_W(\Gamma(M_n(F)))=2(n-1)$.
\end{proof}

\bigskip

 \section{Semisimple rings}

\bigskip

In this section, we examine the Wiener index and the Wiener complexity of a zero divisor graph of a semisimple finite ring. Of course, the Wedderburn-Artin theorem implies that every finite semisimple ring is isomorphic to a direct product of matrix rings over some fields. 
Let us firstly prove the following lemma. 

\begin{lemma}
\label{zerodivisorsproduct}
Let $l \in \NN$ and $R=R_1 \times R_2 \times \ldots \times R_l$ be a direct product of finite rings. Then $Z(R)=R \setminus R_1^* \times R_2^* \times \ldots \times R_l^*$ and $\diam(\Gamma(R)) \leq 3$. Furthermore,
for any $a=(a_1,a_2,\ldots,a_l), b=(b_1,b_2,\ldots,b_l) \in \Gamma(R)$ we have $d(a,b) = 3$ if and only if for every $i \in \{1,2,\ldots,l\}$ we have either $a_i \in R_i^*$ or $b_i \in R_i^*$ and there exists $j \in \{1,2,\ldots,l\}$ such that $a_j \neq 0$ and $b_j \neq 0$.
\end{lemma}
\begin{proof}
 It is well known that in a finite ring every element is either invertible or a zero-divisor. The first statement then follows from the fact that an element $(x_1, x_2, \ldots, x_l) \in R_1 \times R_2 \times \ldots \times R_l$ is invertible if and only if $x_i \in R_i$ is invertible for every $i=1,2,\ldots,l$, while the fact that $\diam(\Gamma(R)) \leq 3$ follows from \cite[Theorem 3.1]{MR2946298}.

 Now, choose arbitrary $a=(a_1,a_2,\ldots,a_l)$ and $b=(b_1,b_2,\ldots,b_l)$ in $Z(R)$. By the above paragraph, there exists $i \in \{1,2,\ldots,l\}$ such that $a_i \in Z(R_i)$. If $b_i \in Z(R_i)$ then there exists a path $a_i-x_i - b_i$ of distance at most 2 in $\Gamma(R_i)$  by Lemma \ref{diam}, so we have a path $a-(0,0,\ldots,0,x_i,0,\ldots,0)-b$ in $\Gamma(R)$, thus $d(a,b) \leq 2$.

Therefore, we can now suppose that for every $i \in \{1,2,\ldots,l\}$ we have either $a_i \in R_i^*$ or $b_i \in R_i^*$.  Obviously, in this case we have $ab=ba=0$ if and only if $a_j=0$ or $b_j=0$ for every $j \in \{1,2,\ldots,l\}$ and in this case $d(a,b)=1$. So, suppose now that there exists $x=(x_1,x_2,\ldots,x_l) \in Z(R)$ such that $a-x-b$ is a path in $\Gamma(R)$. Since for every $i \in \{1,2,\ldots,l\}$ either $a_i$ or $b_i$ is not a zero-divisor, this implies $x_i=0$ for every $i \in \{1,2,\ldots,l\}$, a contradiction. Therefore $d(a,b)=3$ in this case if and only if there exists $j \in \{1,2,\ldots,l\}$ such that $a_j \neq 0$ and $b_j \neq 0$.
\end{proof}

This now gives us the necessary tools to calculate the number of vertices that are at distance $3$ in the zero-divisor graph of a direct product of finite rings. 

\begin{lemma}
\label{distance3directproduct}
Let $l \in \NN$ and $R=R_1 \times R_2 \times \ldots \times R_l$ be a direct product of finite rings. Then $|\{(a,b); a,b \in V(\Gamma(R)) \text { such that } d(a,b)=3\}|=\left( T(R) - 2^l +2 \right)\prod_{i=1}^l{|R_i^*|}$, where $T(R)=\sum_{\Lambda \subseteq \{1,2,\ldots,l\}, |\Lambda| \geq 2}(2^{|\Lambda|}-2)\prod_{i=1}^l{k_i}$, where $k_i=|Z(R_i)|$ if $i \in \Lambda$ and $k_i=|R_i^*|$ otherwise.
\end{lemma}
\begin{proof}
By Lemma \ref{zerodivisorsproduct}, we know that for any elements $a=(a_1,a_2,\ldots,a_l), b=(b_1,b_2,\ldots,b_l) \in \Gamma(R)$ we have $d(a,b) = 3$ if and only if for every $i \in \{1,2,\ldots,l\}$ we have either $a_i \in R_i^*$ or $b_i \in R_i^*$ and there exists $j \in \{1,2,\ldots,l\}$ such that $a_j \neq 0$ and $b_j \neq 0$. 
Denote $\Omega(\Gamma(R))=\{(a,b); a,b \in V(\Gamma(R)) \text { and for every } j \in \{1,2,\ldots,l\}, \text { we have } a_j \in R_j^* \text { or } b_j \in R_j^*\}$. Since $a,b \in V(\Gamma(R))$, at least one of the components of $a$ and $b$ has to be a zero-divisor. Denote $\Lambda_a = \{j \in \{1,2,\ldots,l\}; a_j \text { is a zero-divisor} \}$ and similarly $\Lambda_b = \{j \in \{1,2,\ldots,l\}; b_j \text { is a zero-divisor} \}$. Observe that $\Lambda_a, \Lambda_b \neq \emptyset$ and $\Lambda_a \cap \Lambda_b = \emptyset$. 
It is well known that every element of a finite ring is either a unit or a zero-divisor. This implies that
$a_j \in R_j^*$ for every $j \notin \Lambda_a$ and $b_j \in R_j^*$ for every $j \notin \Lambda_b$. 
Therefore, $|\Omega(\Gamma(R))|=\sum_{\emptyset  \neq \Lambda_a, \Lambda_b \subseteq \{1,2,\ldots,l\},\Lambda_a \cap \Lambda_b = \emptyset}\prod_{i=1}^l{k_i |R_i^*|}$, where $k_i=|Z(R_i)|$ if $i \in \Lambda_a \cup \Lambda_b$ and $k_i=|R_i^*|$ otherwise. Denote $\Lambda=\Lambda_a \cup \Lambda_b$ and observe that $|\Lambda| \geq 2$ and that every $\Lambda$ can be obtained from $\Lambda_a$ and $\Lambda_b$ in $\sum_{t=1}^{|\Lambda| - 1}{|\Lambda| \choose t}=2^{|\Lambda|}-2$ ways. Thus, $|\Omega(\Gamma(R))|=T(R)\prod_{i=1}^l{|R_i^*|}$, for $T(R)=\sum_{\Lambda \subseteq \{1,2,\ldots,l\},|\Lambda| \geq 2}(2^{|\Lambda|}-2)\prod_{i=1}^l{k_i}$, where $k_i=|Z(R_i)|$ if $i \in \Lambda$ and $k_i=|R_i^*|$ otherwise. In order to find the number of elements at distance $3$, we now have to subtract the number of all elements from $|\Omega(\Gamma(R))|$ that are at distance $1$. So, choose $a=(a_1,a_2,\ldots,a_l), b=(b_1,b_2,\ldots,b_l) \in \Gamma(R)$ such that $(a,b) \in \Omega(\Gamma(R))$ and that $d(a,b) = 1$. 
Observe that this time we have  $\Lambda_a, \Lambda_b \neq \emptyset$, $\Lambda_a \cap \Lambda_b = \emptyset$ and $\Lambda_a \cup \Lambda_b = \{1,2,\ldots,l\}$. By observing that the set $\{1,2,\ldots,l\}$ can be obtained as a union of $\Lambda_a$ and $\Lambda_b$ in $\sum_{t=1}^{l - 1}{l \choose t}=2^{l}-2$ ways, we deduce that 
the number of elements from $\Omega(\Gamma(R))$ that are at distance $1$ is equal to $(2^l-2)\prod_{i=1}^l{|R_i^*|}$ and thus the result follows.
\end{proof}

Next, we shall need to calculate the sizes of left and right annihilators of matrices in a simple finite ring. 

\begin{lemma}
\label{annleftright}
Let $F$ be a finite field of cardinality $q$ and $n \geq 1$. Suppose $A \in M_n(F)$ is a matrix of rank $k$ for some $k \in \{0,1,\ldots,n\}$. Then $|\Annl(A)|=|\Annr(A)|=q^{n(n-k)}$ and furthermore $|\Annl(A) \cap \Annr(A)|=q^{(n-k)^2}$.
\end{lemma}
\begin{proof}
Since $A$ is a matrix of rank $k$, there exist invertible matrices (corresponding to a suitable changes in bases) $P, Q \in M_n(F)$ such that $PAQ=\left[\begin{matrix} I_k & 0 \\
0 & 0 \end{matrix}\right]$, where $I_k$ denotes the identity matrix of size $k$. Note that any matrix with the first $k$ rows equal to zero is in $\Annl(PAQ)$, so $|\Annl(PAQ)|=q^{n(n-k)}$. Furthermore, for any matrix $B \in M_n(F)$, we have $BA=0$ if and only if $(BP^{-1})(PAQ)=0$. Since the mapping $B \mapsto BP^{-1}$ is a bijection on $M_n(F)$, we get $|\Annl(A)|=q^{n(n-k)}$. The proof for $|\Annr(A)|$ is symmetrical. Also, for any  $C \in M_n(F)$, we have $CA=AC=0$ if and only if $(Q^{-1}CP^{-1})(PAQ)=(PAQ)(Q^{-1}CP^{-1})=0$. Since the mapping $C \mapsto Q^{-1}CP^{-1}$ is also a bijection on $M_n(F)$, we conclude that $|\Annl(A) \cap \Annr(A)|=|\Annl(PAQ) \cap \Annr(PAQ)|=q^{(n-k)^2}$.
\end{proof}

Now, let us examine the degrees of matrices in $\Gamma(R)$ in the semisimple case.

\begin{lemma}
\label{annsizes}
Let $n_1,n_2, \ldots, n_l \geq 1$ be integers, $F_1, F_2, \ldots, F_l$ finite fields of cardinalities $q_1, q_2, \ldots, q_l$ respectively, and $R=M_{n_1}(F_1) \times M_{n_2}(F_2) \times \ldots \times M_{n_l}(F_l)$. Suppose $A=(A_1,A_2,\ldots,A_l) \in R \setminus \{0\}$ where $A_i$ is of rank $k_i$ for every $i=1,2,\ldots,l$. Then $\degr(A)=\prod_{i=1}^l{q_i^{(n_i-k_i)^2}}  \left(2\prod_{i=1}^l{q_i^{k_i(n_i-k_i)}}-1\right)- \epsilon$, where $\epsilon=1$ if $A^2 \neq 0$ and $\epsilon=2$ if $A^2=0$.
\end{lemma}
\begin{proof}
If $A$ is not a zero-divisor, then $\degr(A)=0$ by definition and note that our formula yields $\degr(A)=1-\epsilon$, where $\epsilon=1$, since $A^2 \neq 0$, therefore the statement holds in this case. 
Suppose therefore that $A$ is a zero-divisor, so there exists $B=(B_1,B_2,\ldots,B_l) \in R$ such that $AB=0$ or $BA=0$. Then $A_iB_i=0$ for every $i \in \{1,2,\ldots,l\}$ or $B_iA_i=0$ for every $i \in \{1,2,\ldots,l\}$. Therefore by the principle of inclusion and exclusion, $\degr(A)=\prod_{i=1}^l{|\Annl(A_i)|}+\prod_{i=1}^l{|\Annr(A_i)|}-\prod_{i=1}^l{|\Annl(A_i) \cap \Annr(A_i)|}-\epsilon$, where $\epsilon=1$ if $A^2 \neq 0$ ($0$ is not a vertex in $\Gamma(R)$) and $\epsilon=2$ if $A^2=0$ (there exists no edge between $A$ and $A$). Lemma \ref{annleftright} now yields 
\begin{multline*}
\degr(A)=2\prod_{i=1}^l{q_i^{n_i(n_i-k_i)}}-\prod_{i=1}^l{q_i^{(n_i-k_i)^2}} - \epsilon= \\
\prod_{i=1}^l{q_i^{(n_i-k_i)^2}}  \left(2\prod_{i=1}^l{q_i^{k_i(n_i-k_i)}}-1\right)- \epsilon.
\end{multline*}
\end{proof}

Now, we have everything ready to calculate the Wiener index in the semisimple case. 

\begin{theorem}
\label{mainsemisimple}
Let $n_1,n_2, \ldots, n_l \geq 1$ be integers, $F_1, F_2, \ldots, F_l$ finite fields of cardinalities $q_1, q_2, \ldots, q_l$ respectively, and $R=M_{n_1}(F_1) \times M_{n_2}(F_2) \times \ldots \times M_{n_l}(F_l)$. Then 
\begin{multline*}
W(\Gamma(R))=|Z(R)|^2 - \frac{5}{2}|Z(R)|+\frac{3}{2} + \\
\left( \frac{T(R)}{2} - 2^{l-1} +1 \right)\prod_{i=1}^l{q_i^{n_i \choose 2}\prod_{j_i=1}^{n_i}(q_i^{j_i}-1)} - \\
\frac{1}{2}\sum_{k_1,\ldots,k_l=0}^{n_1,\ldots,n_l}\left[\prod_{i=1}^l\left(\prod_{j_i=0}^{k_i-1}{\frac{(q_i^{n_i}-q_i^{j_i})^2}{q_i^{k_i}-q_i^{j_i}}}\right){\prod_{i=1}^lq_i^{(n_i-k_i)^2}}  \left(2\prod_{i=1}^l{q_i^{k_i(n_i-k_i)}}-1\right)\right] \\
 +\frac{1}{2}\prod_{i=1}^l{q_i^{n_i^2}} + \frac{1}{2}\prod_{i=1}^l\prod_{j_i=0}^{n_i-1}{(q_i^{n_i}-q_i^{j_i})} + \frac{1}{2}|N_2(R)|,
\end{multline*}
where the following equations hold: \\
$|Z(R)|=\prod_{i=1}^lq_i^{n_i^2}-\prod_{i=1}^lq_i^{n_i \choose 2}\prod_{j_i=1}^{n_i}(q_i^{j_i}-1)$, \\
$|N_2(R)|=\sum_{k_1=0}^{n_1-1}\sum_{k_2=0}^{n_2-1}\ldots\sum_{k_l=0}^{n_l-1}\prod_{i=1}^l\left(\frac{\prod_{j_i=0}^{2k_i-1} \left( q_i^{n_i}- q_i^{j_i}\right )}{q_i^{k_i^2}\prod_{j_i=0}^{k_i-1}\left( q_i^{k_i}-q_i^{j_i} \right )}\right) - 1$
and \\ $T(R)=\sum_{\Lambda \subseteq \{1,2,\ldots,l\}, |\Lambda| \geq 2}(2^{|\Lambda|}-2)\prod_{i=1}^l{\omega_i}$, where we have $\omega_i=q_i^{n_i^2}-q_i^{n_i \choose 2}\prod_{j_i=1}^{n_i}(q_i^{j_i}-1)$ if $i \in \Lambda$ and $\omega_i=q_i^{n_i \choose 2}\prod_{j_i=1}^{n_i}(q_i^{j_i}-1)$ otherwise.
\end{theorem}
\begin{proof}
For $i=1,2,3$ denote $D_i=|\{ (a,b); a,b \in V(\Gamma(R)) \text { such that } d(a,b)=i\}|$ and observe that by Lemma \ref{zerodivisorsproduct}, for every pair of (distinct) vertices $A, B \in V(\Gamma(R))$, we either have $d(A,B)=1$, $d(A,B)=2$ or $d(A,B)=3$, so $D_1+D_2+D_3=|V(\Gamma(R))|(|V(\Gamma(R))|-1)$.
Therefore $W(\Gamma(R))=\frac{1}{2}(D_1+2D_2+3D_3)=\frac{1}{2}(D_1+3D_3+2(|V(\Gamma(R))|(|V(\Gamma(R))|-1)-D_1-D_3))=|Z(R)|^2-3|Z(R)|+2+\frac{1}{2}(D_3-D_1)$.

Furthermore, we apply Lemma \ref{distance3directproduct}, and get $$D_3=\left( T(R) - 2^l +2 \right)\prod_{i=1}^l{|R_i^*|}= \left( T(R) - 2^l +2 \right)\prod_{i=1}^l{q_i^{n_i \choose 2}\prod_{j_i=1}^{n_i}(q_i^{j_i}-1)},$$ 
where $T(R)=\sum_{\Lambda \subseteq \{1,2,\ldots,l\}, |\Lambda| \geq 2}(2^{|\Lambda|}-2)\prod_{i=1}^l{\omega_i}$ and furthermore $\omega_i=q_i^{n_i^2}-q_i^{n_i \choose 2}\prod_{j_i=1}^{n_i}(q_i^{j_i}-1)$ if $i \in \Lambda$ and $\omega_i=q_i^{n_i \choose 2}\prod_{j_i=1}^{n_i}(q_i^{j_i}-1)$ otherwise.
On the other hand, $D_1=\sum_{0 \neq A \in M_n(R)}\degr(A)$, so Lemmas \ref{rankk} and \ref{annsizes} yield (we have to sum over all matrices $A=(A_1,A_2,\ldots,A_l)$, where at least one of the $A_i$ is non-zero and at least one of the $A_i$ is not of full rank)
\begin{multline*}
D_1=\sum_{k_1,\ldots,k_l=0}^{n_1,\ldots,n_l}\left[\prod_{i=1}^l\left(\prod_{j_i=0}^{k_i-1}{\frac{(q_i^{n_i}-q_i^{j_i})^2}{q_i^{k_i}-q_i^{j_i}}}\right){\prod_{i=1}^lq_i^{(n_i-k_i)^2}}  \left(2\prod_{i=1}^l{q_i^{k_i(n_i-k_i)}}-1\right)\right] \\
  -\prod_{i=1}^l{q_i^{n_i^2}} - \prod_{i=1}^l\prod_{j_i=0}^{n_i-1}{(q_i^{n_i}-q_i^{j_i})} - (|Z(R)|-1) - |N_2(R)|,
\end{multline*}
where
$N_2(R)=\{0 \neq A \in R; A^2=0\}$.
Obviously, for $A=(A_1,A_2, \ldots, A_l)$ we have $A^2=0$ if and only if $A_i^2=0$ for every $i=1,2,\ldots,l$.  
With the use of Lemma \ref{sqzero}, we quickly see that 
$$|N_2(R)|=\sum_{k_1,\ldots,k_l=0}^{n_1-1,\ldots,n_l-1}\prod_{i=1}^l\left(\frac{\prod_{j_i=0}^{2k_i-1} \left( q_i^{n_i}- q_i^{j_i}\right )}{q_i^{k_i^2}\prod_{j_i=0}^{k_i-1}\left( q_i^{k_i}-q_i^{j_i} \right )}\right) - 1.$$
If we collect all this together, we finally arrive at
\begin{multline*}
W(\Gamma(R))=|Z(R)|^2 - \frac{5}{2}|Z(R)|+\frac{3}{2} + 
\left( \frac{T(R)}{2} - 2^{l-1} +1 \right)\prod_{i=1}^l{q_i^{n_i \choose 2}\prod_{j_i=1}^{n_i}(q_i^{j_i}-1)} - \\
\frac{1}{2}\sum_{k_1,\ldots,k_l=0}^{n_1,\ldots,n_l}\left[\prod_{i=1}^l\left(\prod_{j_i=0}^{k_i-1}{\frac{(q_i^{n_i}-q_i^{j_i})^2}{q_i^{k_i}-q_i^{j_i}}}\right)\prod_{i=1}^l{q_i^{(n_i-k_i)^2}}  \left(2\prod_{i=1}^l{q_i^{k_i(n_i-k_i)}}-1\right)\right] \\ 
+ \frac{1}{2}\prod_{i=1}^l{q_i^{n_i^2}} 
+ \frac{1}{2}\prod_{i=1}^l\prod_{j_i=0}^{n_i-1}{(q_i^{n_i}-q_i^{j_i})} + \frac{1}{2}|N_2(R)|.
\end{multline*}
Finally, by Lemma \ref{novertices}, we also have $$|Z(R)|=|R|-|R^*|=\prod_{i=1}^lq_i^{n_i^2}-\prod_{i=1}^lq_i^{n_i \choose 2}\prod_{j_i=1}^{n_i}(q_i^{j_i}-1),$$ which proves the assertion of this theorem.
\end{proof}

Let us illustrate the above theorem in a special case of product of two matrix rings of size $2$.

\begin{example}
Let $F_1, F_2$ be finite fields and denote $x=|F_1|$ and $y=|F_2|$.
Then Theorem \ref{mainsemisimple} together with some tedious calculations yield that $W(\Gamma(M_{2}(F_1) \times M_{2}(F_2)))=x^8y^6 + 3x^7y^7 + x^6y^8 + 2x^8y^5+ 2x^5y^8 - x^8y^4 - 9x^7y^5 - x^6y^6 - 9x^5y^7 - x^4y^8 - 2x^8y^3 + 3x^7y^4 - 2x^6y^5 - 2x^5y^6 + 3x^4y^7 - 2x^3y^8 + x^8y^2 + 6x^7y^3 + x^6y^4 + 22x^5y^5 + x^4y^6 + 6x^3y^7 + x^2y^8 - 3x^7y^2 + 2x^6y^3 - 10x^5y^4 - 10x^4y^5 + 2x^3y^6 - 3x^2y^7 - x^6y^2 - 12x^5y^3 - \frac{5}{2}x^4y^4 - 12x^3y^5 -x^2y^6 + 9x^5y^2 + 8x^4y^3 + 8x^3y^4 + 9x^2y^5 - x^5y - x^4y^2 + 6x^3y^3 - x^2y^4 - xy^5 - x^4y - 6x^3y^2 - 6x^2y^3 -xy^4 + \frac{5}{2}x^2y^2 + xy + 1$. 
In particular, we can calculate for example that
$W(\Gamma(M_2(\ZZ_2) \times M_2(\ZZ_2)))=49005$ and $W(\Gamma(M_2(\ZZ_6)))=1089421$.
\end{example}

Finally, let us also examine the Wiener complexity in the semisimple case. Note that this bound is again dependent only on the size of matrices and not on the size of the underlying fields.

\begin{theorem}
\label{complexsemisimple}
Let $n_1,n_2, \ldots, n_l \geq 1$ be integers, $F_1, F_2, \ldots, F_l$ finite fields and $R=M_{n_1}(F_1) \times M_{n_2}(F_2) \times \ldots \times M_{n_l}(F_l)$. Then 
$$C_W(\Gamma(R))\leq \prod_{i=1}^{l}n_i + \prod_{i=1}^{l}(n_i+1) -3.$$
\end{theorem}
\begin{proof}
Choose $0 \neq A \in Z(R)$. Denote $D_3(A)=|\{B \in Z(R); d(A,B)=3 \}|$. Then by Lemma \ref{zerodivisorsproduct}, we have $\Tr(A)=\degr(A)+2(|V(\Gamma(R))|-1 - \degr(A) - D_3(A))+3D_3(A)=2(|Z(R)|-2) - \degr(A) + D_3(A)$. Let $A=(A_1,A_2, \ldots, A_l)$ with $\rk(A_i)=k_i$ for every $i=1,2,\ldots,l$. Suppose firstly that there exists $i \in \{1,2,\ldots,l\}$ such that $1 \leq k_i \leq n_i-1$. Let us denote $q_i=|F_i|$ for $i=1,2,\ldots,l$. Observe that by Lemma \ref{rankk} and Lemma \ref{zerodivisorsproduct}, $D_3(A)=\prod_{i=1}^{l}\omega_i$, where $\omega_i=q_i^{n_i^2}$ if $k_i=n_i$ and $\omega_i=\prod_{j_i=0}^{n_i-1}{\frac{(q_i^{n_i}-q_i^{j_i})^2}{q_i^{n_i}-q_i^{j_i}}}$ otherwise. On the other hand, if $k_i \in \{0,n_i\}$ for every $i \in \{1,2,\ldots,l\}$, then let $\Lambda_1=\{i \in \{1,2,\ldots,l\}; k_i=0\}$ and $\Lambda_2=\{1,2,\ldots,\l\} \setminus \Lambda_1$. Observe that in this case, $D_3(A)=\left(\prod_{i \in \Lambda_1}\prod_{j_i=0}^{n_i-1}{\frac{(q_i^{n_i}-q_i^{j_i})^2}{q_i^{n_i}-q_i^{j_i}}}\right)\left(-1+ \prod_{i \in \Lambda_2}{q_i^{n_i^2}} \right)$.
Furthermore, by Lemma \ref{annsizes}, $\degr(A)=\prod_{i=1}^l{q_i^{(n_i-k_i)^2}}  \left(2\prod_{i=1}^l{q_i^{k_i(n_i-k_i)}}-1\right)- \epsilon$, where $\epsilon=1$ if $A^2 \neq 0$, and $\epsilon=2$ if $A^2=0$.
This implies that $\Tr(A)$ is dependent only on the numbers $k_1,k_2,\ldots,k_l$ and the on the fact whether $A^2$ is zero or not. Since $A \in Z(R) \setminus \{0\}$, at least one of the numbers $k_1,k_2,\ldots,k_l$ has to be non-zero and there has to exist $i \in \{1,2,\ldots,l\}$ such that $k_i \leq n_i-1$. Since $A^2=0$ implies $A_i^2=0$ for every $i \in \{1,2,\ldots,l\}$, and since for every vector $(k_1, k_2, \ldots, k_l) \neq 0^l$ such that $k_i \leq n_i-1$ for every $i$ there exist matrices $B=(B_1,B_2,\ldots,B_l), C=(C_1,C_2,\ldots,C_l) \in Z(R) \setminus \{0\}$ such that $\rk(B_i)=\rk(C_i)=k_i$ for every $i \in \{1,2,\ldots,l\}$ and $B^2=0$, while $C^2 \neq 0$, this implies that $C_W(\Gamma(R)) \leq 2\left(\prod_{i=1}^{l}n_i - 1\right) + S(R)-1$, where $S(R)$ denotes the number of different vectors $(k_1,k_2,\ldots,k_l) \in \NN^l$ such that $k_i \geq 0$ for every $i \in \{1,2,\ldots,l\}$ and there exists $j \in \{1,2,\ldots,l\}$ such that $k_j=n_j$. Let us calculate $S(R)$. Denote $x_i=n_i+1$ for every $i \in \{1,2,\ldots,l\}$ and observe that by the inclusion-exclusion principle we have $S(R)=e_{l-1}(x_1,x_2,\ldots,x_l)-e_{l-2}(x_1,x_2,\ldots,x_l)+\ldots+(-1)^{l-2}e_1(x_1,x_2,\ldots,x_l)+(-1)^{l-1}$, where
$e_i(x_1,x_2,\ldots,x_l)$ denotes the $i$-th elementary symmetrical polynomial in variables $x_1,x_2,\ldots,x_l$. We further calculate that $S(R)=(-1)^{l-1}((-1)^{l-1}e_{l-1}(x_1,x_2,\ldots,x_l)+(-1)^{l-2}e_{l-2}(x_1,x_2,\ldots,x_l)+\ldots+(-1)^{1}e_1(x_1,x_2,\ldots,x_l)+1)=
(-1)^{l-1}((-1)^le_l(x_1,x_2,\ldots,x_l)+ \\ (-1)^{l-1}e_{l-1}(x_1,x_2,\ldots,x_l)+(-1)^{l-2}e_{l-2}(x_1,x_2,\ldots,x_l)+\ldots+ \\ (-1)^{1}e_1(x_1,x_2,\ldots,x_l)+1)+e_l(x_1,x_2,\ldots,x_l)$. By the well known identity for the symmetric polynomials, this yields 
$$S(R)=(-1)^{l-1}\prod_{i=1}^l(1-x_i)+e_l(x_1,x_2,\ldots,x_l)=-\prod_{i=1}^ln_i + \prod_{i=1}^l(n_i+1).$$ This finally gives us
$C_W(\Gamma(R)) \leq \prod_{i=1}^{l}n_i + \prod_{i=1}^{l}(n_i+1) -3$.
\end{proof}

\acknowledgment{The author acknowledges the financial support from the Slovenian Research Agency  (research core funding No. P1-0222).}

\bibliographystyle{amsplain}
\bibliography{biblio}

\end{document}